\documentclass{amsart}
\usepackage{amssymb}
\usepackage{amsfonts}

\newtheorem{theorem}{Theorem}
\theoremstyle{plain}
\newtheorem{acknowledgement}{Acknowledgement}

\newtheorem{corollary}{Corollary}

\newtheorem{definition}{Definition}

\newtheorem{lemma}{Lemma}

\numberwithin{equation}{section}
\input{tcilatex}

\begin{document}
\title[Short Title]{On some strong convergence results of a new Halpern-type
iterative process for quasi-nonexpansive mappings and accretive operators in
Banach spaces}
\author{$^{1}$Kadri DOGAN}
\address{Department of Mathematical Engineering, Yildiz Technical
University, Davutpasa Campus, Esenler, 34210 \.{I}stanbul, Turkey}
\email{dogankadri@hotmail.com}
\author{$^{2}$Vatan KARAKAYA}
\curraddr{[Department of Mathematical Engineering, Yildiz Technical
University, Davutpasa Campus, Esenler, 34210 \.{I}stanbul, Turkey}
\email{vkkaya@yahoo.com}
\subjclass[2010]{ 47H09, 47H10, 37C25}
\keywords{Iterative process, accretive operator, strong convergence; sunny
nonexpansive retraction, uniformly convex Banach space}
\dedicatory{$^{(1,2)}$Department of Mathematical Engineering, Yildiz
Technical University, Davutpasa Campus, Esenler, 34210 \.{I}stanbul, Turkey}

\begin{abstract}
In this study, we introduce a new iterative processes to approximate common
fixed points of an infinite family of quasi-nonexpansive mappings and obtain
a strongly convergent iterative sequence to the common fixed points of these
mappings in a uniformly convex Banach space. Also we prove that this process
approximates to zeros of an infinite family of accretive operators and we
obtain a strong convergence result for these operators. 
\end{abstract}

\maketitle

\section{introduction and preliminaries}

Throughout this study, the set of all non-negative integers and the set of
reel numbers, which we denote by $%
\mathbb{N}
$ and $%
\mathbb{R}
$, respectively.

Geometric properties of Banach spaces and nonlinear algorithms, a topic of
intensive research efforts, in particular within the past 30 years, or so.
Some geometric properties of Banach spaces play a crucial role in fixed
point theory. In the first part of the study, we investigate these geometric
concepts most of which are well known. We begin with some basic notations.

In 1936, Clarkson \cite{Clrk36} achieved a remarkable study on uniform
convexity. It signalled the beginning of extensive research efforts on the
geometry of Banach spaces and its applications. Most of the results
indicated in this work\ were developed in 1991 or later.

Let $C$ be a nonempty, closed and convex set, which is subset of $B$ Banach
space, and let $B^{\ast }$ be the dual space of $B$. We define the modulus
of convexity of $B$, $\delta _{B}(\epsilon )$, as follows:%
\begin{equation*}
\delta _{B}(\epsilon )=\inf \left\{ 1-\frac{\left\Vert a+b\right\Vert }{2}%
:a,b\in \overline{B(0,1)},\left\Vert a-b\right\Vert \geq \epsilon \right\} .
\end{equation*}

The modulus of convexity is a real valued function defined from $[0,2]$ to $%
[0,1]$ which is continuous on $[0,2)$. A Banach space is uniformly convex if
and only if $\ \delta _{B}(\epsilon )>0$ for all $\epsilon >0.$ Let $B$ be a
normed space and $S_{B}=\left\{ a\in B:\left\Vert a\right\Vert =1\right\} $
the unit sphere of $B$. Then norm of $B$ is G\^{a}teaux differentiable at
point $a\in S_{B}$ if for$\ a\in S_{B}$%
\begin{equation*}
\frac{d}{dt}\left( \left\Vert a+tb\right\Vert \right)
|_{t=0}=\lim_{t\rightarrow 0}\frac{\left\Vert a+tb\right\Vert -\left\Vert
a\right\Vert }{t}
\end{equation*}%
exists. The norm of $B$ is said to G\^{a}teaux differentiable if it is G\^{a}%
teaux differentiable at each point of $S_{B}.$In the case, $B$ is called
smooth. The norm of $B$ is said to uniformly G\^{a}teaux differentiable if
for each $b\in S_{B}$, the limit is approached uniformly for $a\in S_{B}$.
Similarly, if the norm of $B$ is \ uniformly G\^{a}teaux differentiable,
then $B$ is called uniformly smooth. A normed space $B$ is called strictly
convex if for all $a,b\in B$, $a\neq b,\left\Vert a\right\Vert =\left\Vert
b\right\Vert =1,$ we have 
\begin{equation*}
\left\Vert \lambda a+\left( 1-\lambda \right) b\right\Vert <1,\text{ \ for
all }\lambda \in \left( 0,1\right) .
\end{equation*}

Now, the result of the above definitions we give the following theorem and
corollary without proofs.

\begin{theorem}
\cite{Chidume09} Let $B$ be a Banach space.

$1)$ $B$ is uniformly convex if and only if $B^{\ast }$ is uniformly smooth.

2) $B$ is uniformly smooth if and only if $B^{\ast }$ is uniformly smooth.
\end{theorem}

\begin{theorem}
\cite{Chidume09} Every uniformly smooth space is reflexive.
\end{theorem}

A self mapping $\phi $ on $\left[ 0,\infty \right) $ is said to be a gauge
map if it is continuous and strictly increasing such that $\phi \left(
0\right) =0$. Let $\phi $ be a gauge function$,$ and let $B$ be any normed
space. If the mapping $J_{\phi }:B\rightarrow 2^{B^{\ast }}$ defined by%
\begin{equation*}
J_{\phi }a=\left\{ f\in B^{\ast }:\left\langle a,f\right\rangle =\left\Vert
a\right\Vert \left\Vert f\right\Vert ;\left\Vert f\right\Vert =\phi \left(
\left\Vert a\right\Vert \right) \right\}
\end{equation*}

for all $a\in B$, then $J_{\phi }$ is said to be the duality map with gauge
function $\phi .$If $\phi \left( t\right) =t$ is selected, then $J_{\phi }=J$
duality mapping is called the normalized duality map.

Let 
\begin{equation*}
\psi \left( t\right) =\int_{0}^{t}\phi \left( \varsigma \right) d\varsigma 
\text{, \ \ }t\geq 0\text{,}
\end{equation*}%
then $\psi \left( \delta t\right) \leq \delta \phi \left( t\right) $ for
each $\delta \in \left( 0,1\right) $.

\begin{equation*}
\rho \left( t\right) =\sup \left\{ \frac{\left\Vert a+b\right\Vert
+\left\Vert a-b\right\Vert }{2}-1:a,b\in B,\left\Vert a\right\Vert =1\text{ }%
and\text{ }\left\Vert b\right\Vert =t\right\}
\end{equation*}%
is called the modulus of smoothness of $B$, where $\rho :\left[ 0,\infty
\right) \rightarrow \left[ 0,\infty \right) $ is a mapping. Also, $%
\lim_{t\rightarrow 0}\frac{\rho \left( t\right) }{t}=0$ if and only if $B$
is uniformly smoothness.

Assume that $q\in 
\mathbb{R}
$ is chosen in the interval $\left( 1,2\right] $. \.{I}f a Banach space $B$
is $q-$uniformly smoothness, then it provides the following conditions. $(i) 
$ there exists a fix $c>0,$ ($ii)$ $\rho \left( t\right) \leq ct^{q}$. For $%
q>2$, there is no q-uniformly smoothness Banach space. In \cite{Cioranescu13}%
, this assertion was showed by Cioranescu. We say that the mapping J is
single-valued and also smoothness if the Banach space $B$ having a
sequentially continuous duality mapping J from weak topology to weak$^{\ast }
$ topology. The space B is said to have weakly sequentially continuous
duality map if \ duality mapping J is continuous and single-valued, see \cite%
{Cioranescu13, Reich92},

Let C be a nonempty subset of Banach space B and $T:C\rightarrow B$ be a
nonself mapping. Also, $\ $let $F\left( T\right) =\left\{ a\in
C:Ta=a\right\} $ denote the set of fixed point of $T$. The map $%
T:C\rightarrow B$ can be referred as follows:

1) It is nonexpansive if $\left\Vert Ta-Tb\right\Vert \leq \left\Vert
a-b\right\Vert $ for all $a,b\in C.$

2) It is quasi-nonexpansive if $\left\Vert Ta-p\right\Vert \leq \left\Vert
a-p\right\Vert $ for all $a\in C$ and $p\in F\left( T\right) $.

In the following iterative process defined by Dogan and Karakaya \cite{dogan}%
.

Let $C$ be a convex subset of a normed space $B$\ and $T:C\rightarrow C$\ \
a self map on $B$. 
\begin{eqnarray}
x_{0} &=&x\in C  \label{it} \\
f\left( T,x_{n}\right) &=&\left( 1-\wp _{n}\right) x_{n}+\xi
_{n}Tx_{n}+\left( \wp _{n}-\xi _{n}\right) Ty_{n}  \notag \\
y_{n} &=&\left( 1-\zeta _{n}\right) x_{n}+\zeta _{n}Tx_{n}  \notag
\end{eqnarray}%
for $\ n\geq 0$, where $\left\{ \xi _{n}\right\} ,~\left\{ \wp _{n}\right\}
, $ $\left\{ \zeta _{n}\right\} $ satisfies the following conditions

$C_{1})$ $\wp _{n}\geq \xi _{n}$

$C_{2})$ $\left\{ \wp _{n}-\xi _{n}\right\} _{n=0}^{\infty },\left\{ \wp
_{n}\right\} _{n=0}^{\infty },\left\{ \zeta _{n}\right\} _{n=0}^{\infty
},\left\{ \xi _{n}\right\} _{n=0}^{\infty }\in \left[ 0,1\right] $

$C_{3})$ $\sum_{n=0}^{\infty }\wp _{n}=\infty .$

In 1967, Halpern \cite{Halpern67} was the first who introduced the following
iteration process under the nonexpansive mapping T. For any initial value $%
a_{0}\in C$ and any fix $u\in C$, $\varphi _{n}\in \left[ 0,1\right] $ such
that $\varphi _{n}=n^{-b}$, 
\begin{equation}
a_{n+1}=\varphi _{n}u+\left( 1-\varphi _{n}\right) Ta_{n}\text{ \ \ \ \ \ }%
\forall n\in 
\mathbb{N}
\text{,}  \label{h}
\end{equation}%
where $b\in \left( 0,1\right) $. In 1977, Lions \cite{Lions77} showed that
the iteration process $\left( \ref{h}\right) $ converges strongly to a fixed
point of T, where $\left\{ \varphi _{n}\right\} _{n\in 
\mathbb{N}
}$ provides the following first three conditions:

$\left( C_{1}\right) $ $\ \lim_{n\rightarrow \infty }\varphi _{n}=0;$

$\left( C_{2}\right) \sum_{n=1}^{\infty }\varphi _{n}=\infty ;$

$\left( C_{3}\right) \ \lim_{n\rightarrow \infty }\frac{\varphi
_{n+1}-\varphi _{n}}{\varphi _{n+1}^{2}}=0$;

$\left( C_{4}\right) $ $\sum_{n=1}^{\infty }\left\vert \varphi
_{n+1}-\varphi _{n}\right\vert <\infty ;$

$\left( C_{5}\right) $ $\lim_{n\rightarrow \infty }\frac{\varphi
_{n+1}-\varphi _{n}}{\varphi _{n+1}}=0$;

$\left( C_{6}\right) $ $\left\vert \varphi _{n+1}-\varphi _{n}\right\vert
\leq o\left( \varphi _{n+1}\right) +\sigma _{n}$, $\sum_{n=1}^{\infty
}\sigma _{n}<\infty .$

Also, by exchanging of the above conditions, several authors were obtained
various results in different spaces. Let us list the main ones as follows:

$\left( 1\right) $ In \cite{Wittmann92}, Wittmann was shown that the
sequence $\left\{ a_{n}\right\} _{n\in 
\mathbb{N}
}$ converges strongly of a fixed point of T by the conditions $C_{1}$, $%
C_{2} $ and $C_{4}$.

$\left( 2\right) $ In \cite{Reich80, Reich94}, Reich was shown that the
sequence $\left\{ a_{n}\right\} _{n\in 
\mathbb{N}
}$ converges strongly of a fixed point of T in the uniformly smooth Banach
spaces by the conditions $C_{1}$, $C_{2}$ and $C_{6}$.

$\left( 3\right) $ In \cite{Shioji97}, Shioji and Takahashi were shown that
the sequence $\left\{ a_{n}\right\} _{n\in 
\mathbb{N}
}$ converges strongly of a fixed point of T in the Banach spaces with
uniformly G\u{a}teaux differentiable norms by the conditions\ $C_{1}$, $%
C_{2} $ and $C_{4}$.

$\left( 4\right) $ In \cite{Xu02}, Xu was shown that he sequence $\left\{
a_{n}\right\} _{n\in 
\mathbb{N}
}$ converges strongly of a fixed point of $T$ by the conditions\ $C_{1}$, $%
C_{2}$ and $C_{5}$. \ \ \ \ \ \ \ \ \ \ \ \ \ \ \ \ \ \ \ \ \ \ \ \ \ \ \ \
\ \ \ \ \ \ \ \ \ \ \ \ \ \ \ \ \ \ \ \ \ \ \ \ \ \ \ \ \ \ \ \ \ \ \ \ \ \
\ \ \ \ \ \ \ \ \ \ \ \ \ \ \ \ \ \ \ \ \ \ \ \ \ \ \ \ \ \ \ \ \ \ \ \ \ \
\ \ \ \ \ \ \ \ \ \ \ \ \ \ \ \ \ \ \ \ \ \ \ \ \ \ \ \ \ \ \ \ \ \ \ \ \ \
\ \ \ \ \ \ \ \ \ \ \ \ \ \ \ \ \ \ \ \ \ \ \ \ \ \ \ \ \ \ \ \ \ \ \ \ \ \
\ \ \ \ \ \ \ \ \ \ \ \ \ \ \ \ \ \ \ \ \ \ \ \ \ \ \ \ \ \ \ \ \ \ \ \ \ \
\ \ \ \ \ \ \ \ \ \ \ \ \ \ \ \ \ \ \ \ \ \ \ \ \ \ \ \ \ \ \ \ \ \ \ \ \ \
\ \ \ \ \ \ \ \ \ \ \ \ \ \ \ \ \ \ \ \ \ \ \ \ \ \ \ \ \ \ \ \ \ \ \ \ \ \
\ \ \ \ \ \ \ 

Are\textbf{\ }the\textbf{\ }conditions $C_{1}$and $C_{2}$ enough to
guarantee the strong convergence of $\left( \ref{h}\right) $ iteration
process for the quasi-nonexpansive mappings, see \cite{Halpern67}?

This question was answered positively by some authors. In the following
list, you can see the work of these authors \cite{Chidume06, Suzuki07, Hu08,
Song09, Saejung10, Nilsrakoo11, Li13, Pang13}. But, in \cite{Suzuki09}, they
were shown that the answer to open question is not positive for nonexpansive
mappings in Hilbert spaces.

The effective domain and range of $A:B\rightarrow 2^{B}$ denoted by $%
dom\left( A\right) =\left\{ a\in B:Aa\neq \varnothing \right\} $ and $%
R\left( A\right) $, respectively. If there exists $j\in J\left( a-b\right) $
such that $\left\langle a-b\text{, }j\right\rangle \geq 0$ and $%
J:B\rightarrow 2^{B^{\ast }}$ duality mapping, then the map $A$ is said to
be accretive, for all $a,b\in B$. If $R\left( I+rA\right) =B$, for each $%
r\geq 0$, then the accretive map $A$ is $m-$accretive operator. All this
paper, let $A:B\rightarrow 2^{B}$ be an accretive operator and be has a
zero. Now, we can define a single- valued mapping such that $J_{r}=\left(
I+rA\right) ^{-1}:B\rightarrow dom\left( A\right) $. It is called the
resolvent of $A$ for $r>0$. Let $A^{-1}=\left\{ a\in B:0\in Aa\right\} $. It
is known that $A^{-1}=F\left( J_{r}\right) $ for all $r>0$, $\left( see\text{%
,\cite{Yao09, Takahashi00} }\right) $.

Let $B$ be a reflexive, smooth and strictly convex Banach space and $C$ be a
nonempty, closed and convex subset $\left( ccs\right) $ of $B$. Under these
conditions, for any $a\in B$, there exists a unique point $z\in C$ such that 
\begin{equation*}
\left\Vert z-a\right\Vert \leq \min_{t\in C}\left\Vert t-a\right\Vert \text{%
; see \ \cite{Takahashi00}.\ \ \ \ \ \ \ \ \ }
\end{equation*}

\begin{definition}
\label{d1}\cite{Takahashi00} If $P_{C}a=z$, then the map $P_{C}:B\rightarrow
C$ is called the metric projection.
\end{definition}

Assume that $a\in B$ and $z\in C$, then $z=P_{C}a$ iff $\left\langle
z-t,J\left( a-z\right) \right\rangle \geq 0$, for all $t\in C$. In a real
Hilbert space $H$, there is a $P_{C}:H\rightarrow C$ projection mapping,
which is nonexpansive, but, such a $P_{C}:B\rightarrow C$ projection mapping
does not provide the nonexpansive property in a Banach space $B$, where C is
a nonempty, closed and convex subset of them; see \cite{Goebel84}.

\begin{definition}
\cite{Reich73} \ Let $C\subset D$ be subsets of Banach space $B$. A mapping $%
Q:C\rightarrow D$ is said to be a sunny if $Q\left( \delta x+\left( 1-\delta
\right) Qx\right) =Qx$, for each $x\in B$ and $\delta \in \left[ 0,1\right) $%
.
\end{definition}

$Q$ is said to be a retraction if and only if $Q^{2}=Q$. $Q$ is a sunny
nonexpansive retraction if and only if it is sunny, nonexpansive and
retraction.

In the next time, we will need lemmas in order to prove the main results.

\begin{lemma}
\label{l1}\cite{Xu02} Let $B$ be a Banach space with weakly sequentially
continuous duality mapping $J_{\phi }$. Then 
\begin{equation*}
\psi \left( \left\Vert a+b\right\Vert \right) \leq \psi \left( \left\Vert
a\right\Vert \right) +2\left\langle b,j_{\phi }\left( a+b\right)
\right\rangle
\end{equation*}

for $a,b\in B$. If we get $J$ instead of $J_{\phi }$, we have 
\begin{equation*}
\left\Vert a+b\right\Vert ^{2}\leq \left\Vert a\right\Vert
^{2}+2\left\langle b,j\left( a+b\right) \right\rangle
\end{equation*}

for $a,b\in B$.
\end{lemma}

\begin{lemma}
\label{l2}\cite{Gossez72} Let $B$ be a Banach space with weakly sequentially
continuous duality mapping $J_{\phi }$ and $C$ be a $ccs$ of $B$. Let $%
T:C\rightarrow C$ be a nonexpansive operator having $F\left( T\right) \neq
\varnothing $. Then, for each $u\in C,$ there exists $a\in F\left( T\right) $
such that%
\begin{equation*}
\left\langle u-a,J\left( b-a\right) \right\rangle \leq 0
\end{equation*}

for all $b\in F\left( T\right) $.
\end{lemma}

\begin{lemma}
\label{l3}\cite{Xu022} Let $B$ be a reflexive Banach space with weakly
sequentially continuous duality mapping $J_{\phi }$ and $C$ be a $ccs$ of $B$%
. Assume that $T:C\rightarrow C$ is a nonexpansive operator. Let $z_{t}\in C$
be the unique solution in $C$ to the equation $z_{t}=tu+\left( 1-t\right)
Tz_{t}$ such that $u\in C$ and $t\in \left( 0,1\right) $. Then $T$ has a
fixed point if and only if $\left\{ z_{t}\right\} _{t\in \left( 0,1\right) }$
remains bounded as $t\rightarrow 0^{+}$, and in this case, $\left\{
z_{t}\right\} _{t\in \left( 0,1\right) }$ converges as $t\rightarrow 0^{+}$
strongly to fixed point of $T$. If we get the sunny nonexpansive retraction
defined by $Q:C\rightarrow F\left( T\right) $ such that%
\begin{equation*}
Q\left( u\right) =\lim_{t\rightarrow 0}z_{t}\text{,}
\end{equation*}

then $Q\left( u\right) $ solves the variational inequality%
\begin{equation*}
\left\langle u-Q\left( u\right) ,J_{\phi }\left( b-Q\left( u\right) \right)
\right\rangle \leq 0\text{, }u\in C~and~b\in F\left( T\right) \text{.}
\end{equation*}
\end{lemma}

One of the useful and remarkable results in the theory of nonexpansive
mappings is demiclosed principle. It is defined as follows.

\begin{definition}
\cite{Pang13} Let $B$ be a Banach space, C a nonempty subset of $B$, and $\
T:C\rightarrow B$ a mapping. Then the mapping $T$ is said to be demiclosed
at origin, that is, for any sequence $\{a_{n}\}_{n\in N}$ in $C$ which $%
a_{n}\rightharpoonup $ $p$ and $\left\Vert Ta_{n}-a_{n}\right\Vert
\rightarrow 0$ imply that $Tp=p$.
\end{definition}

\begin{lemma}
\label{l4}\cite{Browder68} Let $B$ be a reflexive Banach space having weakly
sequentially continuous duality mapping $J_{\phi }$ with a gauge function $%
\phi $, $C$ be a $ccs$ of $B$ and $T:C\rightarrow B$ be a nonexpansive
mapping. Then $I-T$ is demiclosed at each $p\in B$, i.e., for any sequence $%
\{a_{n}\}_{n\in N}$ in $C$ which converges weakly to $a$, and $(I-T)a_{n}$ $%
\rightarrow p$ converges strongly imply that ($I-T)a=p$. (Here $I$ is the
identity operator of $B$ into itself.) In particular, assuming $p=0$, it is
obtained $a\in F\left( T\right) $.
\end{lemma}

\begin{lemma}
\label{l5}\cite{Park94} Let $\left\{ \mu _{n}\right\} _{n\in 
\mathbb{N}
}$ be a nonnegative real sequence and satisfies the following inequality%
\begin{equation*}
\mu _{n+1}\leq \left( 1-\varphi _{n}\right) \mu _{n}+\varphi _{n}\epsilon
_{n}\text{,}
\end{equation*}

and assume that $\left\{ \varphi _{n}\right\} _{n\in 
\mathbb{N}
}$ and $\left\{ \epsilon _{n}\right\} _{n\in 
\mathbb{N}
}$ satisfy the following conditions:

$\left( 1\right) $ $\left\{ \varphi _{n}\right\} _{n\in 
\mathbb{N}
}\subset \left[ 0,1\right] $ and $\dsum\limits_{n=1}^{\infty }\varphi
_{n}=\infty $,

$\left( 2\right) $ $\lim \sup_{n\rightarrow \infty }\epsilon _{n}\leq 0$, or

$\left( 3\right) $ $\dsum\limits_{n=1}^{\infty }\varphi _{n}\epsilon
_{n}<\infty $,

then $\lim_{n\rightarrow \infty }\mu _{n}=0$.
\end{lemma}

\begin{lemma}
\label{l6}\cite{Takahashi00} Let $\ B$ be a real Banach space, and let $A$
be an $m-$accretive operator on $\ B$. For $t>0$, let $J_{t}$ be a resolvent
operator related to $A$ and $t$. Then%
\begin{equation*}
\left\Vert J_{k}a-J_{l}a\right\Vert \leq \frac{\left\vert k-l\right\vert }{k}%
\left\Vert a-J_{k}a\right\Vert \text{, for all }k,l>0\text{ and }a\in B\text{%
.}
\end{equation*}
\end{lemma}

\begin{lemma}
\label{l7}\cite{Mainge08} Let $\left\{ \mu _{n}\right\} _{n\in 
\mathbb{N}
}$ be a sequence of real numbers such that there exists a subsequence $%
\left\{ \mu _{n_{i}}\right\} _{i\in 
\mathbb{N}
}$ of $\left\{ \mu _{n}\right\} _{n\in 
\mathbb{N}
}$ which satisfies $\mu _{n_{i}}<\mu _{n_{i+1}}$ for all $i\geq 0$. Also, we
consider a subsequence $\left\{ \eta _{\left( n\right) }\right\} _{n\geq
n_{0}}\subset 
\mathbb{N}
$ defined by%
\begin{equation*}
\eta _{\left( n\right) }=\max \left\{ k\leq n:\mu _{k}\leq \mu
_{k+1}\right\} \text{.}
\end{equation*}

Then $\left\{ \eta _{\left( n\right) }\right\} _{n\geq n_{0}}$ is a
nondecreasing sequence providing $\lim_{n\rightarrow \infty }\eta _{\left(
n\right) }=\infty $, for all $n\geq n_{0}$. Hence, it holds that $\mu _{\eta
_{\left( n\right) }}\leq \mu _{\eta _{\left( n\right) +1}}$, and implies
that $\mu _{n}\leq \mu _{\eta _{\left( n\right) +1}}$.
\end{lemma}

\begin{lemma}
\label{l8}\cite{Chang10} Let $B$ be a uniformly convex Banach space and \ $%
t>0$ be a constant. Then there exists a continuous, strictly increasing and
convex function $g:\left[ 0,2t\right) \rightarrow \left[ 0,\infty \right) $
such that%
\begin{equation*}
\left\Vert \dsum\limits_{i=1}^{\infty }\rho _{i}a_{i}\right\Vert ^{2}\leq
\dsum\limits_{i=1}^{\infty }\rho _{i}\left\Vert a_{i}\right\Vert ^{2}-\rho
_{k}\rho _{l}g\left( \left\Vert a_{k}-a_{l}\right\Vert \right)
\end{equation*}

$\forall k,l\geq 0$, $a_{i}\in B_{t}=\left\{ z\in B:\left\Vert z\right\Vert
\leq t\right\} $, $\rho _{i}\in \left( 0,1\right) $ and $i\geq 0$ with $%
\dsum\limits_{i=0}^{\infty }\rho _{i}=1$.
\end{lemma}

\section{Main results}

\begin{theorem}
Let $B$ be a real uniformly convex Banach space having the normalized
duality mapping $J$ and $C$ be a $ccs$ of $B$. Assume that $\left\{
T_{i}\right\} _{i\in 
\mathbb{N}
\cup \left\{ 0\right\} }$ is a infinite family of quasi nonexpansive
mappings given in the form $T_{i}:C\rightarrow C$ such that $%
F=\dbigcap\limits_{i=0}^{\infty }F\left( T_{i}\right) \neq \varnothing $,
and for each $i\geq 0$, $T_{i}-I$ is demiclosed at zero. Let $\left\{
v_{n}\right\} _{n\in 
\mathbb{N}
}$ be a sequence generated by%
\begin{equation}
\left\{ 
\begin{array}{c}
v_{1}\text{, }u\in C\text{ arbitrarily chosen, \ \ \ \ \ \ \ \ \ \ \ \ \ \ \
\ \ \ \ \ \ \ } \\ 
v_{n+1}=\xi _{n}u+\left( 1-\zeta _{n}\right) T_{0}v_{n}+\left( \zeta
_{n}-\xi _{n}\right) T_{0}w_{n}\text{ \ \ } \\ 
w_{n}=\varphi _{n,0}v_{n}+\dsum\limits_{i=1}^{\infty }\varphi
_{n,i}T_{i}v_{n}\text{, \ }\ n\geq 0\text{,\ \ \ \ \ \ \ \ \ }%
\end{array}%
\right.  \label{itnew}
\end{equation}

where$\left\{ \zeta _{n}\right\} _{n\in 
\mathbb{N}
}$, $\left\{ \xi _{n}\right\} _{n\in 
\mathbb{N}
}$ and $\left\{ \varphi _{n,i}\right\} _{n\in 
\mathbb{N}
,i\in 
\mathbb{N}
\cup \left\{ 0\right\} }$ are sequences in $\left[ 0,1\right] $ satisfying
the following control conditions:

\qquad $\left( 1\right) ~\lim_{n\rightarrow \infty }\xi _{n}=0$;

\qquad $\left( 2\right) ~\dsum\limits_{n=1}^{\infty }\xi _{n}=\infty $;

\qquad $\left( 3\right) ~\varphi _{n,0}+\dsum\limits_{i=1}^{\infty }\varphi
_{n,i}=1$, for all $n\in 
\mathbb{N}
$;

\qquad $\left( 4\right) ~\lim \inf_{n\rightarrow \infty }\zeta _{n}\varphi
_{n,0}\varphi _{n,i}>0$, for all $n\in 
\mathbb{N}
$.

Then $\left\{ v_{n}\right\} _{n\in 
\mathbb{N}
}$ converges strongly as $n\rightarrow \infty $ to $P_{F}u$, where the map $%
P_{F}:B\rightarrow F$ is the metric projection.
\end{theorem}

\begin{proof}
The proof consists of three parts.

Step 1. Prove that $\left\{ v_{n}\right\} _{n\in 
\mathbb{N}
}$, $\left\{ w_{n}\right\} _{n\in 
\mathbb{N}
}$ and $\left\{ T_{i}v_{n}\right\} _{n\in 
\mathbb{N}
,i\in 
\mathbb{N}
\cup \left\{ 0\right\} }$ are bounded. Firstly, we show that $\left\{
v_{n}\right\} _{n\in 
\mathbb{N}
}$ is bounded. Let $p\in F$ be fixed. By Lemma \ref{l8}, we have the
following inequality%
\begin{eqnarray}
\left\Vert w_{n}-p\right\Vert ^{2} &=&\left\Vert \varphi
_{n,0}v_{n}+\dsum\limits_{i=1}^{\infty }\varphi
_{n,i}T_{i}v_{n}-p\right\Vert ^{2}  \notag \\
&\leq &\varphi _{n,0}\left\Vert v_{n}-p\right\Vert
^{2}+\dsum\limits_{i=1}^{\infty }\varphi _{n,i}\left\Vert
T_{i}v_{n}-p\right\Vert ^{2}-\varphi _{n,0}\varphi _{n,i}g\left( \left\Vert
v_{n}-T_{i}v_{n}\right\Vert \right)  \notag \\
&\leq &\varphi _{n,0}\left\Vert v_{n}-p\right\Vert
^{2}+\dsum\limits_{i=1}^{\infty }\varphi _{n,i}\left\Vert v_{n}-p\right\Vert
^{2}-\varphi _{n,0}\varphi _{n,i}g\left( \left\Vert
v_{n}-T_{i}v_{n}\right\Vert \right)  \notag \\
&=&\left\Vert v_{n}-p\right\Vert ^{2}-\varphi _{n,0}\varphi _{n,i}g\left(
\left\Vert v_{n}-T_{i}v_{n}\right\Vert \right)  \label{1} \\
&\leq &\left\Vert v_{n}-p\right\Vert ^{2}\text{.}  \notag
\end{eqnarray}

This show that 
\begin{eqnarray*}
\left\Vert v_{n+1}-p\right\Vert &=&\left\Vert \xi _{n}u+\left( 1-\zeta
_{n}\right) T_{0}v_{n}+\left( \zeta _{n}-\xi _{n}\right) T_{0}w_{n}\text{ }%
-p\right\Vert \\
&\leq &\xi _{n}\left\Vert u-p\right\Vert +\left( 1-\zeta _{n}\right)
\left\Vert T_{0}v_{n}-p\right\Vert +\left( \zeta _{n}-\xi _{n}\right)
\left\Vert T_{0}w_{n}\text{ }-p\right\Vert \\
&\leq &\xi _{n}\left\Vert u-p\right\Vert +\left( 1-\zeta _{n}\right)
\left\Vert v_{n}-p\right\Vert +\left( \zeta _{n}-\xi _{n}\right) \left\Vert
w_{n}\text{ }-p\right\Vert \\
&\leq &\xi _{n}\left\Vert u-p\right\Vert +\left( 1-\xi _{n}\right)
\left\Vert v_{n}-p\right\Vert \\
&\leq &\max \left\{ \left\Vert u-p\right\Vert \text{, }\left\Vert
v_{n}-p\right\Vert \text{ }\right\}
\end{eqnarray*}

If we continue the way of induction, we have%
\begin{equation*}
\left\Vert v_{n+1}-p\right\Vert \leq \max \left\{ \left\Vert u-p\right\Vert 
\text{, }\left\Vert v_{1}-p\right\Vert \text{ }\right\} \text{, }\forall
n\in 
\mathbb{N}
\text{.}
\end{equation*}

Therefore, we conclude that $\left\Vert v_{n+1}-p\right\Vert $ is bounded,
this implies that $\left\{ v_{n}\right\} _{n\in 
\mathbb{N}
}$ is bounded. Furthermore, it is easily show that $\left\{
T_{i}v_{n}\right\} _{n\in 
\mathbb{N}
,i\in 
\mathbb{N}
\cup \left\{ 0\right\} }$ $\ $and $\left\{ w_{n}\right\} _{n\in 
\mathbb{N}
}$ are bounded too.

Step 2. Show that for any $n\in 
\mathbb{N}
$,%
\begin{equation}
\left\Vert v_{n+1}-z\right\Vert ^{2}\leq \left( 1-\xi _{n}\right) \left\Vert
v_{n}-z\right\Vert ^{2}+2\xi _{n}\left\langle u-z,J\left( v_{n+1}-z\right)
\right\rangle \text{.}  \label{2}
\end{equation}

By considering $\left( \ref{1}\right) $, we have%
\begin{equation}
\left\Vert w_{n}-z\right\Vert ^{2}=\left\Vert v_{n}-z\right\Vert
^{2}-\varphi _{n,0}\varphi _{n,i}g\left( \left\Vert
v_{n}-T_{i}v_{n}\right\Vert \right) \text{.}  \label{3}
\end{equation}

$\left( \ref{3}\right) $ implies that%
\begin{eqnarray}
\left\Vert v_{n+1}-z\right\Vert ^{2} &=&\left\Vert \xi _{n}u+\left( 1-\zeta
_{n}\right) T_{0}v_{n}+\left( \zeta _{n}-\xi _{n}\right) T_{0}w_{n}\text{ }%
-z\right\Vert ^{2}  \notag \\
&\leq &\xi _{n}\left\Vert u-z\right\Vert ^{2}+\left( 1-\zeta _{n}\right)
\left\Vert T_{0}v_{n}-z\right\Vert ^{2}+\left( \zeta _{n}-\xi _{n}\right)
\left\Vert T_{0}w_{n}\text{ }-z\right\Vert ^{2}  \notag \\
&\leq &\xi _{n}\left\Vert u-z\right\Vert ^{2}+\left( 1-\zeta _{n}\right)
\left\Vert v_{n}-z\right\Vert ^{2}  \label{4} \\
&&+\left( \zeta _{n}-\xi _{n}\right) \left[ \left\Vert v_{n}-z\right\Vert
^{2}-\varphi _{n,0}\varphi _{n,i}g\left( \left\Vert
v_{n}-T_{i}v_{n}\right\Vert \right) \right] \\
&=&\xi _{n}\left\Vert u-z\right\Vert ^{2}+\left( 1-\xi _{n}\right)
\left\Vert v_{n}-z\right\Vert ^{2}  \notag \\
&&-\zeta _{n}\varphi _{n,0}\varphi _{n,i}g\left( \left\Vert
v_{n}-T_{i}v_{n}\right\Vert \right) +\xi _{n}\varphi _{n,0}\varphi
_{n,i}g\left( \left\Vert v_{n}-T_{i}v_{n}\right\Vert \right) \text{.}
\end{eqnarray}

Assume that $K_{1}=\sup \left\{ \left\vert \left\Vert u-z\right\Vert
^{2}-\left\Vert v_{n}-z\right\Vert ^{2}\right\vert +\xi _{n}\varphi
_{n,0}\varphi _{n,i}g\left( \left\Vert v_{n}-T_{i}v_{n}\right\Vert \right)
\right\} $. It is conclude form $\left( \ref{4}\right) $ that 
\begin{equation}
\zeta _{n}\varphi _{n,0}\varphi _{n,i}g\left( \left\Vert
v_{n}-T_{i}v_{n}\right\Vert \right) \leq \left\Vert v_{n}-z\right\Vert
^{2}-\left\Vert v_{n+1}-z\right\Vert ^{2}+\xi _{n}K_{1}\text{.}  \label{5}
\end{equation}

By Lemma \ref{l1} and $\left( \ref{1}\right) $, we have%
\begin{eqnarray*}
\left\Vert v_{n+1}-z\right\Vert ^{2} &=&\left\Vert \xi _{n}u+\left( 1-\zeta
_{n}\right) T_{0}v_{n}+\left( \zeta _{n}-\xi _{n}\right) T_{0}w_{n}\text{ }%
-z\right\Vert ^{2} \\
&=&\left\Vert \xi _{n}\left( u-z\right) +\left( 1-\zeta _{n}\right) \left(
T_{0}v_{n}-z\right) +\left( \zeta _{n}-\xi _{n}\right) \left( T_{0}w_{n}%
\text{ }-z\right) \right\Vert ^{2} \\
&\leq &\left\Vert \left( 1-\zeta _{n}\right) \left( T_{0}v_{n}-z\right)
+\left( \zeta _{n}-\xi _{n}\right) \left( T_{0}w_{n}\text{ }-z\right)
\right\Vert ^{2} \\
&&+2\left\langle \xi _{n}\left( u-z\right) ,J\left( v_{n+1}-z\right)
\right\rangle \\
&\leq &\left( 1-\zeta _{n}\right) \left\Vert T_{0}v_{n}-z\right\Vert
^{2}+\left( \zeta _{n}-\xi _{n}\right) \left\Vert T_{0}w_{n}\text{ }%
-z\right\Vert ^{2} \\
&&+2\left\langle \xi _{n}\left( u-z\right) ,J\left( v_{n+1}-z\right)
\right\rangle \\
&\leq &\left( 1-\zeta _{n}\right) \left\Vert v_{n}-z\right\Vert ^{2}+\left(
\zeta _{n}-\xi _{n}\right) \left\Vert w_{n}\text{ }-z\right\Vert ^{2} \\
&&+2\xi _{n}\left\langle u-z,J\left( v_{n+1}-z\right) \right\rangle \\
&\leq &\left( 1-\zeta _{n}\right) \left\Vert v_{n}-z\right\Vert ^{2}+\left(
\zeta _{n}-\xi _{n}\right) \left\Vert v_{n}-z\right\Vert ^{2} \\
&&+2\xi _{n}\left\langle u-z,J\left( v_{n+1}-z\right) \right\rangle \\
&=&\left( 1-\xi _{n}\right) \left\Vert v_{n}-z\right\Vert ^{2}+2\xi
_{n}\left\langle u-z,J\left( v_{n+1}-z\right) \right\rangle \text{.}
\end{eqnarray*}

Step 3. We show that $v_{n}\rightarrow z$ as $n\rightarrow \infty $.

For this step, we will examine two cases.

Case 1. Suppose that there exists $n_{0}\in 
\mathbb{N}
$ such that $\left\{ \left\Vert v_{n}-z\right\Vert \right\} _{n\geq n_{0}}$
is nonincreasing. furthermore, the sequence $\left\{ \left\Vert
v_{n}-z\right\Vert \right\} _{n\in 
\mathbb{N}
}$ is convergent. Thus, it is clear that $\left\Vert v_{n}-z\right\Vert
^{2}-\left\Vert v_{n+1}-z\right\Vert ^{2}\rightarrow 0$ as $n\rightarrow
\infty $. In view of condition $\left( 4\right) $ and $\left( \ref{5}\right) 
$, we have%
\begin{equation*}
\lim_{n\rightarrow \infty }g\left( \left\Vert v_{n}-T_{i}v_{n}\right\Vert
\right) =0\text{.}
\end{equation*}%
From the properties of g, we have 
\begin{equation*}
\lim_{n\rightarrow \infty }\left\Vert v_{n}-T_{i}v_{n}\right\Vert =0\text{.}
\end{equation*}%
Also, we can construct the sequences $\left( w_{n}-v_{n}\right) $ and $%
\left( v_{n+1}-w_{n}\right) $, as follows:%
\begin{eqnarray}
w_{n}-v_{n} &=&\varphi _{n,0}v_{n}+\dsum\limits_{i=1}^{\infty }\varphi
_{n,i}T_{i}v_{n}-v_{n}  \notag \\
&=&\dsum\limits_{i=1}^{\infty }\varphi _{n,i}\left( T_{i}v_{n}-v_{n}\right) 
\text{,}  \label{6}
\end{eqnarray}

and%
\begin{equation*}
v_{n+1}-w_{n}=\xi _{n}u+\left( 1-\zeta _{n}\right) T_{0}v_{n}+\left( \zeta
_{n}-\xi _{n}\right) T_{0}w_{n}\text{ }-w_{n}
\end{equation*}%
\begin{eqnarray}
\left\Vert v_{n+1}-w_{n}\right\Vert &=&\left\Vert \xi _{n}\left(
u-T_{0}w_{n}\right) +\zeta _{n}\left( T_{0}v_{n}-T_{0}w_{n}\right) +\left(
T_{0}v_{n}\text{ }-w_{n}\right) \right\Vert  \notag \\
&\leq &\xi _{n}\left\Vert u-T_{0}w_{n}\right\Vert +\zeta _{n}\left\Vert
T_{0}v_{n}-T_{0}w_{n}\right\Vert +\left\Vert T_{0}v_{n}\text{ }%
-w_{n}\right\Vert  \notag \\
&\leq &\xi _{n}\left\Vert u-T_{0}w_{n}\right\Vert +\zeta _{n}\left\Vert
v_{n}-w_{n}\right\Vert +\left\Vert T_{0}v_{n}\text{ }-w_{n}\right\Vert \text{%
.}  \label{7}
\end{eqnarray}

These imply that%
\begin{equation}
\lim_{n\rightarrow \infty }\left\Vert v_{n+1}-w_{n}\right\Vert =0\text{ \ \
\ \ \ \ \ \ \ \ \ \ and \ \ \ \ }\lim_{n\rightarrow \infty }\left\Vert
w_{n}-v_{n}\text{\ }\right\Vert =0\text{.\ \ \ \ }  \label{k}
\end{equation}

By the expressions in $\left( \ref{k}\right) $, we obtain%
\begin{equation*}
\left\Vert v_{n+1}-v_{n}\right\Vert \leq \left\Vert w_{n}-v_{n}\text{\ }%
\right\Vert +\left\Vert v_{n+1}-w_{n}\right\Vert \text{.}
\end{equation*}

This implies that%
\begin{equation}
\lim_{n\rightarrow \infty }\left\Vert v_{n+1}-v_{n}\right\Vert =0\text{.}
\label{8}
\end{equation}

Previously, we have shown that the sequence $\left\{ v_{n}\right\} _{n\in 
\mathbb{N}
}$ is bounded. Therefore, there exists a subsequence $\left\{
v_{n_{j}}\right\} _{j\in 
\mathbb{N}
}$ of $\left\{ v_{n}\right\} _{n\in 
\mathbb{N}
}$ such that $v_{n_{j}+1}\rightarrow l$ for all $j\in 
\mathbb{N}
$. By principle of demiclosedness at zero, It is concluded that $l\in F$.
Considering the above facts and Definition $\left( \ref{d1}\right) $, we
obtain%
\begin{eqnarray}
\limsup_{n\rightarrow \infty }\left\langle u-z,J\left( v_{n+1},z\right)
\right\rangle &=&\lim_{j\rightarrow \infty }\left\langle u-z,J\left(
v_{n_{j}+1}-z\right) \right\rangle  \notag \\
&=&\left\langle u-z,J\left( l-z\right) \right\rangle  \label{9} \\
&=&\left\langle u-P_{F}u,J\left( l-P_{F}u\right) \right\rangle  \notag \\
&\leq &0\text{.}  \notag
\end{eqnarray}

By Lemma $\left( \ref{l5}\right) $, \ we have the desired \ result.

Case 2. Let $\left\{ n_{j}\right\} _{j\in 
\mathbb{N}
}$ be subsequence of $\left\{ n\right\} _{n\in 
\mathbb{N}
}$ such that 
\begin{equation*}
\left\Vert v_{n_{j}}-z\right\Vert \leq \left\Vert v_{n_{j}+1}-z\right\Vert 
\text{, for all }j\in 
\mathbb{N}
\text{.}
\end{equation*}

Then, in view of Lemma $\left( \ref{l7}\right) $, there exists a
nondecreasing sequence $\left\{ m_{k}\right\} _{k\in 
\mathbb{N}
}\subset 
\mathbb{N}
$, and hence%
\begin{equation*}
\left\Vert z-v_{m_{k}}\right\Vert <\left\Vert z-v_{m_{k}+1}\right\Vert \text{
\ \ \ \ \ and \ \ \ }\left\Vert z-v_{k}\right\Vert \leq \left\Vert
z-v_{m_{k}+1}\right\Vert \text{, }\forall k\in 
\mathbb{N}
\text{.\ \ \ }
\end{equation*}

If we rewrite the equation $\left( \ref{5}\right) $ for this Lemma, we have 
\begin{eqnarray*}
\zeta _{m_{k}}\varphi _{m_{k},0}\varphi _{m_{k},i}g\left( \left\Vert
v_{m_{k}}-T_{i}v_{m_{k}}\right\Vert \right) &\leq &\left\Vert
v_{m_{k}}-z\right\Vert ^{2}-\left\Vert v_{m_{k}+1}-z\right\Vert ^{2}+\xi
_{m_{k}}K_{1} \\
&\leq &\xi _{m_{k}}K_{1}\text{, }\forall k\in 
\mathbb{N}
\text{.}
\end{eqnarray*}

Considering the conditions $\left( 1\right) $ and $\left( 2\right) $, we
obtain%
\begin{equation*}
\lim_{k\rightarrow \infty }g\left( \left\Vert
v_{m_{k}}-T_{i}v_{m_{k}}\right\Vert \right) =0\text{.}
\end{equation*}

It follows that%
\begin{equation*}
\lim_{k\rightarrow \infty }\left\Vert v_{m_{k}}-T_{i}v_{m_{k}}\right\Vert =0%
\text{.}
\end{equation*}

Therefore, using the same argument as Case 1, we have%
\begin{equation*}
\limsup_{n\rightarrow \infty }\left\langle u-z,J\left( v_{m_{k}},z\right)
\right\rangle =\limsup_{n\rightarrow \infty }\left\langle u-z,J\left(
v_{v_{m_{k}}+1},z\right) \right\rangle \leq 0\text{.}
\end{equation*}

Using $\left( \ref{2}\right) $, we get%
\begin{equation*}
\left\Vert v_{m_{k}+1}-z\right\Vert ^{2}\leq \left( 1-\xi _{m_{k}}\right)
\left\Vert v_{m_{k}}-z\right\Vert ^{2}+2\xi _{m_{k}}\left\langle u-z,J\left(
v_{m_{k}+1}-z\right) \right\rangle \text{.}
\end{equation*}

Previously, we have shown that the inequality $\left\Vert
v_{m_{k}}-z\right\Vert \leq \left\Vert v_{m_{k}+1}-z\right\Vert $ is
performed, and hence%
\begin{eqnarray*}
\xi _{m_{k}}\left\Vert v_{m_{k}}-z\right\Vert ^{2} &\leq &\left\Vert
v_{m_{k}}-z\right\Vert ^{2}-\left\Vert v_{m_{k}+1}-z\right\Vert ^{2}+2\xi
_{m_{k}}\left\langle u-z,J\left( v_{m_{k}+1}-z\right) \right\rangle \\
&\leq &2\xi _{m_{k}}\left\langle u-z,J\left( v_{m_{k}+1}-z\right)
\right\rangle \text{.}
\end{eqnarray*}

Hence, we get%
\begin{equation}
\lim_{k\rightarrow \infty }\left\Vert v_{m_{k}}-z\right\Vert =0\text{.}
\label{10}
\end{equation}

considering the expressions $\left( \ref{9}\right) $ and $\left( \ref{10}%
\right) $, we obtain%
\begin{equation*}
\lim_{k\rightarrow \infty }\left\Vert v_{m_{k}+1}-z\right\Vert =0\text{.}
\end{equation*}

Finaly, we get $\left\Vert v_{k}-z\right\Vert \leq \left\Vert
v_{m_{k}+1}-z\right\Vert $, $\forall k\in 
\mathbb{N}
$.\ It follows that $v_{m_{k}}\rightarrow z$ as $k\rightarrow \infty $. Then
we have $v_{k}\rightarrow z$ as $n\rightarrow \infty $.
\end{proof}

We obtain the following corollary for a single mapping.

\begin{corollary}
Let $B$ be a real uniformly convex Banach space having the normalized
duality mapping $J$ and $C$ be a $ccs$ of $B$. Assume that $T$ is a quasi
nonexpansive mappings given in the form $T:C\rightarrow C$ and $F$ is set of
fixed point of $T$ and, $T-I$ is demiclosed at zero. Let $\left\{
v_{n}\right\} _{n\in 
\mathbb{N}
}$ be a sequence generated by%
\begin{equation*}
\left\{ 
\begin{array}{c}
v_{1}\text{, }u\in C\text{ arbitrarily chosen, \ \ \ \ \ \ \ \ \ \ \ \ \ \ \
\ \ \ \ \ \ \ } \\ 
v_{n+1}=\xi _{n}u+\left( 1-\zeta _{n}\right) Tv_{n}+\left( \zeta _{n}-\xi
_{n}\right) Tw_{n}\text{ \ \ } \\ 
w_{n}=\left( 1-\varphi _{n}\right) v_{n}+\varphi _{n}Tv_{n}\text{, \ }\
n\geq 0\text{,\ \ \ \ \ \ \ \ \ }%
\end{array}%
\right.
\end{equation*}%
where$\left\{ \zeta _{n}\right\} _{n\in 
\mathbb{N}
}$, $\left\{ \xi _{n}\right\} _{n\in 
\mathbb{N}
}$ and $\left\{ \varphi _{n}\right\} _{n\in 
\mathbb{N}
}$ are sequences in $\left[ 0,1\right] $ satisfying the following control
conditions:

$\qquad \left( 1\right) ~\lim_{n\rightarrow \infty }\xi _{n}=0$;

$\qquad \left( 2\right) ~\dsum\limits_{n=1}^{\infty }\xi _{n}=\infty $;

$\qquad \left( 3\right) ~\lim \inf_{n\rightarrow \infty }\zeta _{n}\left(
1-\varphi _{n}\right) \varphi _{n}>0$, for all $n\in 
\mathbb{N}
$.

Then $\left\{ v_{n}\right\} _{n\in 
\mathbb{N}
}$ converges strongly as $n\rightarrow \infty $ to $P_{F}u$, where the map $%
P_{F}:B\rightarrow F$ is the metric projection.
\end{corollary}

\begin{theorem}
Let $B$ be a real uniformly convex Banach space having the weakly
sequentially continuous duality mapping $J_{\phi }$ and $C$ be a $ccs$ of $B$
such that $\overline{D(A_{i})}\subset C\subset \dbigcap\limits_{r>0}^{\infty
}R(I+rA_{i})$ for each $i\in N$. Assume that $\left\{ A_{i}\right\} _{i\in 
\mathbb{N}
\cup \left\{ 0\right\} }$ is an infinite family of accretive operators
satisfying the range condition, and $r_{n}>0$ and $r>0$ be such that $%
lim_{n\rightarrow \infty }r_{n}=r$. Let $%
J_{r_{n}}^{A_{i}}=(I+r_{n}A_{i})^{-1}$ be the resolvent of $A$. Let $\left\{
v_{n}\right\} _{n\in 
\mathbb{N}
}$ be a sequence generated by%
\begin{equation}
\left\{ 
\begin{array}{c}
v_{1}\text{, }u\in C\text{ arbitrarily chosen, \ \ \ \ \ \ \ \ \ \ \ \ \ \ \
\ \ \ \ \ \ \ } \\ 
v_{n+1}=\xi _{n}u+\left( 1-\zeta _{n}\right) J_{r_{n}}^{A_{0}}v_{n}+\left(
\zeta _{n}-\xi _{n}\right) J_{r_{n}}^{A_{0}}w_{n}\text{ \ \ } \\ 
w_{n}=\varphi _{n,0}v_{n}+\dsum\limits_{i=1}^{\infty }\varphi
_{n,i}J_{r_{n}}^{A_{i}}v_{n}\text{, \ }\ n\geq 0\text{,\ \ \ \ \ \ \ \ \ }%
\end{array}%
\right.  \label{res}
\end{equation}

where$\left\{ \zeta _{n}\right\} _{n\in 
\mathbb{N}
}$, $\left\{ \xi _{n}\right\} _{n\in 
\mathbb{N}
}$ and $\left\{ \varphi _{n,i}\right\} _{n\in 
\mathbb{N}
,i\in 
\mathbb{N}
\cup \left\{ 0\right\} }$ are sequences in $\left[ 0,1\right] $ satisfying
the following control conditions:

\qquad $\left( 1\right) ~\lim_{n\rightarrow \infty }\xi _{n}=0$;

\qquad $\left( 2\right) ~\dsum\limits_{n=1}^{\infty }\xi _{n}=\infty $;

\qquad $\left( 3\right) ~\varphi _{n,0}+\dsum\limits_{i=1}^{\infty }\varphi
_{n,i}=1$, for all $n\in 
\mathbb{N}
$;

\qquad $\left( 4\right) ~\lim \inf_{n\rightarrow \infty }\zeta _{n}\varphi
_{n,0}\varphi _{n,i}>0$, for all $n\in 
\mathbb{N}
$.

If $Q_{Z}:B\rightarrow Z$ is the sunny nonexpansive retraction such that $%
Z=\dbigcap\limits_{i=1}^{\infty }A_{i}^{-1}\left( 0\right) \neq \varnothing $%
, then $\left\{ v_{n}\right\} _{n\in 
\mathbb{N}
}$ converges strongly as $n\rightarrow \infty $ to $Q_{Z}u$.
\end{theorem}

\begin{proof}
The proof consists of three parts.

We note that $Z$ is closed and convex. Set $z=Q_{Z}u$.

Step 1. Prove that $\left\{ v_{n}\right\} _{n\in 
\mathbb{N}
}$, $\left\{ w_{n}\right\} _{n\in 
\mathbb{N}
}$ and $\left\{ J_{r_{n}}^{A_{i}}v_{n}\right\} _{n\in 
\mathbb{N}
,i\in 
\mathbb{N}
\cup \left\{ 0\right\} }$ are bounded. Firstly, we show that $\left\{
v_{n}\right\} _{n\in 
\mathbb{N}
}$ is bounded. Let $p\in Z$ be fixed. By Lemma \ref{l8}, we have the
following inequality%
\begin{eqnarray}
\left\Vert w_{n}-p\right\Vert ^{2} &=&\left\Vert \varphi
_{n,0}v_{n}+\dsum\limits_{i=1}^{\infty }\varphi
_{n,i}J_{r_{n}}^{A_{i}}v_{n}-p\right\Vert ^{2}  \notag \\
&\leq &\varphi _{n,0}\left\Vert v_{n}-p\right\Vert
^{2}+\dsum\limits_{i=1}^{\infty }\varphi _{n,i}\left\Vert
J_{r_{n}}^{A_{i}}v_{n}-p\right\Vert ^{2}-\varphi _{n,0}\varphi _{n,i}g\left(
\left\Vert v_{n}-J_{r_{n}}^{A_{i}}v_{n}\right\Vert \right)  \notag \\
&\leq &\varphi _{n,0}\left\Vert v_{n}-p\right\Vert
^{2}+\dsum\limits_{i=1}^{\infty }\varphi _{n,i}\left\Vert v_{n}-p\right\Vert
^{2}-\varphi _{n,0}\varphi _{n,i}g\left( \left\Vert
v_{n}-J_{r_{n}}^{A_{i}}v_{n}\right\Vert \right)  \notag \\
&=&\left\Vert v_{n}-p\right\Vert ^{2}-\varphi _{n,0}\varphi _{n,i}g\left(
\left\Vert v_{n}-J_{r_{n}}^{A_{i}}v_{n}\right\Vert \right)  \notag \\
&\leq &\left\Vert v_{n}-p\right\Vert ^{2}\text{.}  \label{j1}
\end{eqnarray}

This show that 
\begin{eqnarray*}
\left\Vert v_{n+1}-p\right\Vert &=&\left\Vert \xi _{n}u+\left( 1-\zeta
_{n}\right) J_{r_{n}}^{A_{0}}v_{n}+\left( \zeta _{n}-\xi _{n}\right)
J_{r_{n}}^{A_{0}}w_{n}\text{ }-p\right\Vert \\
&\leq &\xi _{n}\left\Vert u-p\right\Vert +\left( 1-\zeta _{n}\right)
\left\Vert J_{r_{n}}^{A_{0}}v_{n}-p\right\Vert +\left( \zeta _{n}-\xi
_{n}\right) \left\Vert J_{r_{n}}^{A_{0}}w_{n}\text{ }-p\right\Vert \\
&\leq &\xi _{n}\left\Vert u-p\right\Vert +\left( 1-\zeta _{n}\right)
\left\Vert v_{n}-p\right\Vert +\left( \zeta _{n}-\xi _{n}\right) \left\Vert
w_{n}\text{ }-p\right\Vert \\
&\leq &\xi _{n}\left\Vert u-p\right\Vert +\left( 1-\xi _{n}\right)
\left\Vert v_{n}-p\right\Vert \\
&\leq &\max \left\{ \left\Vert u-p\right\Vert \text{, }\left\Vert
v_{n}-p\right\Vert \text{ }\right\}
\end{eqnarray*}

If we continue the way of induction, we have%
\begin{equation*}
\left\Vert v_{n+1}-p\right\Vert =\max \left\{ \left\Vert u-p\right\Vert 
\text{, }\left\Vert v_{1}-p\right\Vert \text{ }\right\} \text{, }\forall
n\in 
\mathbb{N}
\text{.}
\end{equation*}

Therefore, we conclude that $\left\Vert v_{n+1}-p\right\Vert $ is bounded,
this implies that $\left\{ v_{n}\right\} _{n\in 
\mathbb{N}
}$ is bounded. Furthermore, it is easily show that $\left\{
J_{r_{n}}^{A_{i}}v_{n}\right\} _{n\in 
\mathbb{N}
,i\in 
\mathbb{N}
\cup \left\{ 0\right\} }\ $and $\left\{ w_{n}\right\} _{n\in 
\mathbb{N}
}$ are bounded too.

Step 2. Show that for any $n\in 
\mathbb{N}
$,%
\begin{equation}
\left\Vert v_{n+1}-z\right\Vert ^{2}\leq \left( 1-\xi _{n}\right) \left\Vert
v_{n}-z\right\Vert ^{2}+2\xi _{n}\left\langle u-z,J_{\phi }\left(
v_{n+1}-z\right) \right\rangle \text{.}  \label{j2}
\end{equation}

By considering $\left( \ref{j1}\right) $, we have%
\begin{equation}
\left\Vert w_{n}-z\right\Vert ^{2}=\left\Vert v_{n}-z\right\Vert
^{2}-\varphi _{n,0}\varphi _{n,i}g\left( \left\Vert
v_{n}-J_{r_{n}}^{A_{i}}v_{n}\right\Vert \right) \text{.}  \label{j3}
\end{equation}

$\left( \ref{j3}\right) $ implies that%
\begin{eqnarray}
\left\Vert v_{n+1}-z\right\Vert ^{2} &=&\left\Vert \xi _{n}u+\left( 1-\zeta
_{n}\right) J_{r_{n}}^{A_{0}}v_{n}+\left( \zeta _{n}-\xi _{n}\right)
J_{r_{n}}^{A_{0}}w_{n}\text{ }-z\right\Vert ^{2}  \notag \\
&\leq &\xi _{n}\left\Vert u-z\right\Vert ^{2}+\left( 1-\zeta _{n}\right)
\left\Vert J_{r_{n}}^{A_{0}}v_{n}-z\right\Vert ^{2}+\left( \zeta _{n}-\xi
_{n}\right) \left\Vert J_{r_{n}}^{A_{0}}w_{n}\text{ }-z\right\Vert ^{2} 
\notag \\
&\leq &\xi _{n}\left\Vert u-z\right\Vert ^{2}+\left( 1-\zeta _{n}\right)
\left\Vert v_{n}-z\right\Vert ^{2}+\left( \zeta _{n}-\xi _{n}\right) \left[
\left\Vert v_{n}-z\right\Vert ^{2}-\varphi _{n,0}\varphi _{n,i}g\left(
\left\Vert v_{n}-J_{r_{n}}^{A_{i}}v_{n}\right\Vert \right) \right]
\label{j4} \\
&=&\xi _{n}\left\Vert u-z\right\Vert ^{2}+\left( 1-\xi _{n}\right)
\left\Vert v_{n}-z\right\Vert ^{2}-\zeta _{n}\varphi _{n,0}\varphi
_{n,i}g\left( \left\Vert v_{n}-J_{r_{n}}^{A_{i}}v_{n}\right\Vert \right)
+\xi _{n}\varphi _{n,0}\varphi _{n,i}g\left( \left\Vert
v_{n}-J_{r_{n}}^{A_{i}}v_{n}\right\Vert \right) \text{.}  \notag
\end{eqnarray}

Assume that $K_{2}=\sup \left\{ \left\vert \left\Vert u-z\right\Vert
^{2}-\left\Vert v_{n}-z\right\Vert ^{2}\right\vert +\xi _{n}\varphi
_{n,0}\varphi _{n,i}g\left( \left\Vert
v_{n}-J_{r_{n}}^{A_{i}}v_{n}\right\Vert \right) \right\} $. It is conclude
form $\left( \ref{j4}\right) $ that 
\begin{equation}
\zeta _{n}\varphi _{n,0}\varphi _{n,i}g\left( \left\Vert
v_{n}-J_{r_{n}}^{A_{i}}v_{n}\right\Vert \right) \leq \left\Vert
v_{n}-z\right\Vert ^{2}-\left\Vert v_{n+1}-z\right\Vert ^{2}+\xi _{n}K_{2}%
\text{.}  \label{j5}
\end{equation}

By Lemma \ref{l1} and $\left( \ref{j1}\right) $, we have%
\begin{eqnarray*}
\left\Vert v_{n+1}-z\right\Vert ^{2} &=&\left\Vert \xi _{n}u+\left( 1-\zeta
_{n}\right) J_{r_{n}}^{A_{0}}v_{n}+\left( \zeta _{n}-\xi _{n}\right)
J_{r_{n}}^{A_{0}}w_{n}\text{ }-z\right\Vert ^{2} \\
&=&\left\Vert \xi _{n}\left( u-z\right) +\left( 1-\zeta _{n}\right) \left(
J_{r_{n}}^{A_{0}}v_{n}-z\right) +\left( \zeta _{n}-\xi _{n}\right) \left(
J_{r_{n}}^{A_{0}}w_{n}\text{ }-z\right) \right\Vert ^{2} \\
&\leq &\left\Vert \left( 1-\zeta _{n}\right) \left(
J_{r_{n}}^{A_{0}}v_{n}-z\right) +\left( \zeta _{n}-\xi _{n}\right) \left(
J_{r_{n}}^{A_{0}}w_{n}\text{ }-z\right) \right\Vert ^{2}+2\left\langle \xi
_{n}\left( u-z\right) ,J_{\phi }\left( v_{n+1}-z\right) \right\rangle \\
&\leq &\left( 1-\zeta _{n}\right) \left\Vert
J_{r_{n}}^{A_{0}}v_{n}-z\right\Vert ^{2}+\left( \zeta _{n}-\xi _{n}\right)
\left\Vert J_{r_{n}}^{A_{0}}w_{n}\text{ }-z\right\Vert ^{2}+2\left\langle
\xi _{n}\left( u-z\right) ,J_{\phi }\left( v_{n+1}-z\right) \right\rangle \\
&\leq &\left( 1-\zeta _{n}\right) \left\Vert v_{n}-z\right\Vert ^{2}+\left(
\zeta _{n}-\xi _{n}\right) \left\Vert w_{n}\text{ }-z\right\Vert ^{2}+2\xi
_{n}\left\langle u-z,J_{\phi }\left( v_{n+1}-z\right) \right\rangle \\
&\leq &\left( 1-\zeta _{n}\right) \left\Vert v_{n}-z\right\Vert ^{2}+\left(
\zeta _{n}-\xi _{n}\right) \left\Vert v_{n}-z\right\Vert ^{2}+2\xi
_{n}\left\langle u-z,J_{\phi }\left( v_{n+1}-z\right) \right\rangle \\
&=&\left( 1-\xi _{n}\right) \left\Vert v_{n}-z\right\Vert ^{2}+2\xi
_{n}\left\langle u-z,J_{\phi }\left( v_{n+1}-z\right) \right\rangle \text{.}
\end{eqnarray*}

Step 3. We show that $v_{n}\rightarrow z$ as $n\rightarrow \infty $.

For this step, we will examine two cases.

Case 1. Suppose that there exists $n_{0}\in 
\mathbb{N}
$ such that $\left\{ \left\Vert v_{n}-z\right\Vert \right\} _{n\geq n_{0}}$
is nonincreasing. furthermore, the sequence $\left\{ \left\Vert
v_{n}-z\right\Vert \right\} _{n\in 
\mathbb{N}
}$ is convergent. Thus, it is clear that $\left\Vert v_{n}-z\right\Vert
^{2}-\left\Vert v_{n+1}-z\right\Vert ^{2}\rightarrow 0$ as $n\rightarrow
\infty $. In view of condition $\left( 4\right) $ and $\left( \ref{j5}%
\right) $, we have%
\begin{equation*}
\lim_{n\rightarrow \infty }g\left( \left\Vert
v_{n}-J_{r_{n}}^{A_{i}}v_{n}\right\Vert \right) =0\text{.}
\end{equation*}%
From the properties of $g$, we have 
\begin{equation*}
\lim_{n\rightarrow \infty }\left\Vert
v_{n}-J_{r_{n}}^{A_{i}}v_{n}\right\Vert =0\text{.}
\end{equation*}%
Also, we can construct the sequences $\left( w_{n}-v_{n}\right) $ and $%
\left( v_{n+1}-w_{n}\right) $, as follows:%
\begin{eqnarray}
w_{n}-v_{n} &=&\varphi _{n,0}v_{n}+\dsum\limits_{i=1}^{\infty }\varphi
_{n,i}J_{r_{n}}^{A_{i}}v_{n}-v_{n}  \notag \\
&=&\dsum\limits_{i=1}^{\infty }\varphi _{n,i}\left(
J_{r_{n}}^{A_{i}}v_{n}-v_{n}\right) \text{,}  \label{j6}
\end{eqnarray}

and%
\begin{equation*}
v_{n+1}-w_{n}=\xi _{n}u+\left( 1-\zeta _{n}\right)
J_{r_{n}}^{A_{0}}v_{n}+\left( \zeta _{n}-\xi _{n}\right)
J_{r_{n}}^{A_{0}}w_{n}\text{ }-w_{n}
\end{equation*}%
\begin{eqnarray}
\left\Vert v_{n+1}-w_{n}\right\Vert &=&\left\Vert \xi _{n}\left(
u-w_{n}\right) +\left( 1-\zeta _{n}\right) \left(
J_{r_{n}}^{A_{0}}v_{n}-w_{n}\right) +\left( \zeta _{n}-\xi _{n}\right)
\left( J_{r_{n}}^{A_{0}}w_{n}\text{ }-w_{n}\right) \right\Vert  \notag \\
&\leq &\xi _{n}\left\Vert u-w_{n}\right\Vert +\left( 1-\zeta _{n}\right)
\left\Vert J_{r_{n}}^{A_{0}}v_{n}-w_{n}\right\Vert +\left( \zeta _{n}-\xi
_{n}\right) \left\Vert J_{r_{n}}^{A_{0}}w_{n}\text{ }-w_{n}\right\Vert . 
\notag
\end{eqnarray}

These imply that%
\begin{equation}
\lim_{n\rightarrow \infty }\left\Vert v_{n+1}-w_{n}\right\Vert =0\text{ \ \
\ \ \ \ \ \ \ \ \ \ and \ \ \ \ }\lim_{n\rightarrow \infty }\left\Vert
w_{n}-v_{n}\text{\ }\right\Vert =0\text{.\ \ \ \ }  \label{m}
\end{equation}

By the expressions in $\left( \ref{m}\right) $, we obtain%
\begin{equation*}
\left\Vert v_{n+1}-v_{n}\right\Vert \leq \left\Vert w_{n}-v_{n}\text{\ }%
\right\Vert +\left\Vert v_{n+1}-w_{n}\right\Vert \text{.}
\end{equation*}

This implies that 
\begin{equation}
\lim_{n\rightarrow \infty }\left\Vert v_{n+1}-v_{n}\right\Vert =0\text{.}
\label{j8}
\end{equation}

By Lemma \ref{l6} and $\left( \ref{j6}\right) $, we have%
\begin{equation*}
\left\Vert v_{n}-J_{r}^{A_{i}}v_{n}\right\Vert \leq \left\Vert
v_{n}-J_{r_{n}}^{A_{i}}v_{n}\right\Vert +\left\Vert
J_{r_{n}}^{A_{i}}v_{n}-J_{r}^{A_{i}}v_{n}\right\Vert \leq \left\Vert
v_{n}-J_{r_{n}}^{A_{i}}v_{n}\right\Vert +\frac{\left\vert r_{n}-r\right\vert 
}{r_{n}}\left\Vert v_{n}-J_{r_{n}}^{A_{i}}v_{n}\right\Vert .
\end{equation*}

This implies that 
\begin{equation*}
\lim_{n\rightarrow \infty }\left\Vert v_{n}-J_{r}^{A_{i}}v_{n}\right\Vert =0%
\text{, for all }i\in 
\mathbb{N}
.
\end{equation*}

Previously, we have shown that the sequence $\left\{ v_{n}\right\} _{n\in 
\mathbb{N}
}$ is bounded. Therefore, there exists a subsequence $\left\{
v_{n_{j}}\right\} _{j\in 
\mathbb{N}
}$ of $\left\{ v_{n}\right\} _{n\in 
\mathbb{N}
}$ such that $v_{n_{j}+1}\rightarrow l\in F\left( J_{r}^{A_{i}}v_{n}\right) $
for all $j\in 
\mathbb{N}
$. This, together with Lemma \ref{l1} implies that%
\begin{eqnarray}
\limsup_{n\rightarrow \infty }\left\langle u-z,J_{\phi }\left(
v_{n+1},z\right) \right\rangle &=&\lim_{k\rightarrow \infty }\left\langle
u-z,J_{\phi }\left( v_{n_{j}+1}-z\right) \right\rangle  \notag \\
&=&\left\langle u-z,J_{\phi }\left( l-z\right) \right\rangle  \label{j9} \\
&\leq &0\text{.}  \notag
\end{eqnarray}

By Lemma $\left( \ref{l5}\right) $, we obtain the desired \ result.

Case 2. Let $\left\{ n_{j}\right\} _{j\in 
\mathbb{N}
}$ be subsequence of $\left\{ n\right\} _{n\in 
\mathbb{N}
}$ such that 
\begin{equation*}
\left\Vert v_{n_{j}}-z\right\Vert \leq \left\Vert v_{n_{j}+1}-z\right\Vert 
\text{, for all }j\in 
\mathbb{N}
\text{.}
\end{equation*}

Then, in view of Lemma $\left( \ref{l7}\right) $, there exists a
nondecreasing sequence $\left\{ m_{k}\right\} _{k\in 
\mathbb{N}
}\subset 
\mathbb{N}
$, and hence%
\begin{equation*}
\left\Vert z-v_{m_{k}}\right\Vert <\left\Vert z-v_{m_{k}+1}\right\Vert \text{
\ \ \ \ \ and \ \ \ }\left\Vert z-v_{k}\right\Vert \leq \left\Vert
z-v_{m_{k}+1}\right\Vert \text{, }\forall k\in 
\mathbb{N}
\text{.\ \ \ }
\end{equation*}

If we rewrite the equation $\left( \ref{5}\right) $ for this Lemma, we have 
\begin{eqnarray*}
\zeta _{m_{k}}\varphi _{m_{k},0}\varphi _{m_{k},i}g\left( \left\Vert
v_{m_{k}}-J_{r_{n}}^{A_{i}}v_{m_{k}}\right\Vert \right) &\leq &\left\Vert
v_{m_{k}}-z\right\Vert ^{2}-\left\Vert v_{m_{k}+1}-z\right\Vert ^{2}+\xi
_{m_{k}}K_{2} \\
&\leq &\xi _{m_{k}}K_{2}\text{, }\forall k\in 
\mathbb{N}
\text{.}
\end{eqnarray*}

Considering the conditions $\left( 1\right) $ and $\left( 2\right) $, we
obtain%
\begin{equation*}
\lim_{k\rightarrow \infty }g\left( \left\Vert
v_{m_{k}}-J_{r_{n}}^{A_{i}}v_{m_{k}}\right\Vert \right) =0\text{.}
\end{equation*}

It follows that%
\begin{equation*}
\lim_{k\rightarrow \infty }\left\Vert
v_{m_{k}}-J_{r_{n}}^{A_{i}}v_{m_{k}}\right\Vert =0\text{.}
\end{equation*}

Therefore, using the same argument as Case 1, we have%
\begin{equation*}
\limsup_{n\rightarrow \infty }\left\langle u-z,J_{\phi }\left(
v_{m_{k}},z\right) \right\rangle =\limsup_{n\rightarrow \infty }\left\langle
u-z,J_{\phi }\left( v_{v_{m_{k}}+1},z\right) \right\rangle \leq 0\text{.}
\end{equation*}

Using $\left( \ref{j2}\right) $, we get%
\begin{equation*}
\left\Vert v_{m_{k}+1}-z\right\Vert ^{2}\leq \left( 1-\xi _{m_{k}}\right)
\left\Vert v_{m_{k}}-z\right\Vert ^{2}+2\xi _{m_{k}}\left\langle u-z,J_{\phi
}\left( v_{m_{k}+1}-z\right) \right\rangle \text{.}
\end{equation*}

Previously, we have shown that the inequality $\left\Vert
v_{m_{k}}-z\right\Vert \leq \left\Vert v_{m_{k}+1}-z\right\Vert $ is
performed, and hence%
\begin{eqnarray*}
\xi _{m_{k}}\left\Vert v_{m_{k}}-z\right\Vert ^{2} &\leq &\left\Vert
v_{m_{k}}-z\right\Vert ^{2}-\left\Vert v_{m_{k}+1}-z\right\Vert ^{2}+2\xi
_{m_{k}}\left\langle u-z,J_{\phi }\left( v_{m_{k}+1}-z\right) \right\rangle
\\
&\leq &2\xi _{m_{k}}\left\langle u-z,J_{\phi }\left( v_{m_{k}+1}-z\right)
\right\rangle \text{.}
\end{eqnarray*}

Hence, we get%
\begin{equation}
\lim_{k\rightarrow \infty }\left\Vert v_{m_{k}}-z\right\Vert =0\text{.}
\label{j10}
\end{equation}

Considering the expressions $\left( \ref{j9}\right) $ and $\left( \ref{j10}%
\right) $, we obtain%
\begin{equation*}
\lim_{k\rightarrow \infty }\left\Vert v_{m_{k}+1}-z\right\Vert =0\text{.}
\end{equation*}

Finally, we get $\left\Vert v_{k}-z\right\Vert \leq \left\Vert
v_{m_{k}+1}-z\right\Vert $, $\forall k\in 
\mathbb{N}
$.\ It follows that $v_{m_{k}}\rightarrow z$ as $k\rightarrow \infty $. Then
we have $v_{k}\rightarrow z$ as $n\rightarrow \infty $.
\end{proof}

\begin{theorem}
Let $B$ be a real uniformly convex Banach space having a G\^{a}teaux
differentiable norm. and $C$ be a $ccs$ of $B$ such that $\overline{D(A_{i})}%
\subset C\subset \dbigcap\limits_{r>0}^{\infty }R(I+rA_{i})$ for each $i\in
N $. Assume that $\left\{ A_{i}\right\} _{i\in 
\mathbb{N}
\cup \left\{ 0\right\} }$ is an infinite family of accretive operators
satisfying the range condition, and $r_{n}>0$ and $r>0$ be such that $%
lim_{n\rightarrow \infty }r_{n}=r$. Let $%
J_{r_{n}}^{A_{i}}=(I+r_{n}A_{i})^{-1}$ be the resolvent of $A$. Let $\left\{
v_{n}\right\} _{n\in 
\mathbb{N}
}$ be a sequence generated by%
\begin{equation}
\left\{ 
\begin{array}{c}
v_{1}\text{, }u\in C\text{ arbitrarily chosen, \ \ \ \ \ \ \ \ \ \ \ \ \ \ \
\ \ \ \ \ \ \ } \\ 
v_{n+1}=\xi _{n}u+\left( 1-\zeta _{n}\right) J_{r_{n}}^{A_{0}}v_{n}+\left(
\zeta _{n}-\xi _{n}\right) J_{r_{n}}^{A_{0}}w_{n}\text{ \ \ } \\ 
w_{n}=\varphi _{n,0}v_{n}+\dsum\limits_{i=1}^{\infty }\varphi
_{n,i}J_{r_{n}}^{A_{i}}v_{n}\text{, \ }\ n\geq 0\text{,\ \ \ \ \ \ \ \ \ }%
\end{array}%
\right.
\end{equation}

where$\left\{ \zeta _{n}\right\} _{n\in 
\mathbb{N}
}$, $\left\{ \xi _{n}\right\} _{n\in 
\mathbb{N}
}$ and $\left\{ \varphi _{n,i}\right\} _{n\in 
\mathbb{N}
,i\in 
\mathbb{N}
\cup \left\{ 0\right\} }$ are sequences in $\left[ 0,1\right] $ satisfying
the following control conditions:

\qquad $\left( 1\right) ~\lim_{n\rightarrow \infty }\xi _{n}=0$;

\qquad $\left( 2\right) ~\dsum\limits_{n=1}^{\infty }\xi _{n}=\infty $;

\qquad $\left( 3\right) ~\varphi _{n,0}+\dsum\limits_{i=1}^{\infty }\varphi
_{n,i}=1$, for all $n\in 
\mathbb{N}
$;

\qquad $\left( 4\right) ~\lim \inf_{n\rightarrow \infty }\zeta _{n}\varphi
_{n,0}\varphi _{n,i}>0$, for all $n\in 
\mathbb{N}
$.

If $Q_{Z}:B\rightarrow Z$ is the sunny nonexpansive retraction such that $%
Z=\dbigcap\limits_{i=1}^{\infty }A_{i}^{-1}\left( 0\right) \neq \varnothing $%
, then $\left\{ v_{n}\right\} _{n\in 
\mathbb{N}
}$ converges strongly as $n\rightarrow \infty $ to $Q_{Z}u$.
\end{theorem}

\begin{acknowledgement}
Two authors would like to thank Yildiz Technical University Scientific
Research Projects Coordination Department under project number BAPK
2014-07-03-DOP02 for financial support during the preparation of this
manuscript.
\end{acknowledgement}

\end{document}